\documentclass{article}
\usepackage{eurosym}
\usepackage{amsfonts}
\usepackage{amsmath}

\newtheorem{theorem}{Theorem}

\newtheorem{definition}[theorem]{Definition}

\newtheorem{lemma}[theorem]{Lemma}

\newtheorem{remark}[theorem]{Remark}

\newenvironment{proof}[1][Proof]{\noindent\textbf{#1.} }{\ \rule{0.5em}{0.5em}}
\input{tcilatex}
\begin{document}

\title{Mean-field delayed BSDEs in finite and infinite horizon}
\author{\ Nacira AGRAM \thanks{
University Med Khider, Po. Box 145, Biskra $\left( 07000\right) $ Algeria.
Email: agramnacira@yahoo.fr } \thanks{%
Nacira Agram wants to thank the rector of the University of Biskra,
Professor Belkacem Selatnia, for making it possible for her to study outside
of Algeria. } and Elin Engen R{\normalsize \o }se \thanks{%
Department of mathematics, University of Oslo, Box 1053 Blindern, N-0316
Oslo, Norway. Email: elinero@math.uio.no}\ \  \and \textit{Dedicated to
Belkacem Selatnia\ for his birthday\ }}
\date{\ 29 September 2015}
\maketitle

\begin{abstract}
We establish sufficient conditions for the existence and uniqueness of
different types of delayed BSDEs in finite time horizon. We consider then
infinite horizon, replacing the terminal value condition in the finite
horizon case with a condition of strong decay at infinity.
\end{abstract}


\section{Introduction}

Recently, Lasry and Lions in \cite{LL}, introduced a mathematical mean-field
approach for high dimensional systems that involve a large number of
particles (if we deal with statistical mechanics, physics, quantum mechanics
and quantum chemistry) or agents (if we deal with economics, finance and
game theory). After them, a lot of works have been done in mean-field
problems especially in optimal control and games theory. In what follows, we
have to consider a probability space $\left( \Omega ,\mathcal{F},P\right) $
with the filtration $(\mathcal{F}_{t})_{t\geq 0},$ generated by the Brownian
motion $B$ and an independent compensated Poisson random measure $\tilde{N}%
(dt,d\zeta ).$

In the paper \cite{HOS}, the authors consider a general stochastic optimal
control of mean-field problems where the mean-field term appears as a
function $F:L^{2}(P)\rightarrow
\mathbb{R}
,$ which $F$ is a Fr\'{e}chet differentiable in $L^{2}(P).$ This includes
the case of a mean-field function $F(x)=F(P(x))$ where the function $F$ is
continuously differentiable w.r.t. the measure, this case was studied in the
paper \cite{CD} and the references there in.

Buckdahn et al in \cite{B1} studied a mean-field backward stochastic
differential equation (MFBSDE) driven by a forward stochastic differential
equation (FSDE) of Mackean-Vlasov type. They proved that the triplet $%
(X^{N},Y^{N},Z^{N})$ of $N$ independent copies, converges in law to the
solution of some forward-backward stochastic differential equation of
mean-field type, which is not only governed by a Brownian motion but also by
an independent Gaussian field.

On the other hand, if we want to find an investment strategy and an
investment portfolio which should replicate a liability or meet a target
which depend on the applied strategy or the portfolio depends on its past
values. In this setting, a managed investment portfolio serves
simultaneously as the underlying security on which the liability/target is
contingent and as a replicating portfolio for that liability/target. This is
usually the case for capital-protected investments and performance-linked
pay-offs. The delayed backward stochastic differential equations (DBSDEs)
are the best tool to solve this financial problems. DBSDE can also arise in
variable annuities, unit-linked products and participating contract. For
more applications of such equations, we refer to \cite{D1} and \cite{D2}.

For the more, a possible application of BSDE should be a recursive utility
associated to the consumption rate $"\pi ".$ It was introduced by Duffie and
Epstien in \cite{DE} where they proposed to use the recursive utility
process $Y$ as a part of a solution associated to the following BSDE in the
Brownian case:%
\begin{equation}
dY(t)=-f(t,Y(t),Z(t),\pi (t))dt+Z(t)dB(t),\text{ \ \ \ \ \ }t\in \left[ 0,T%
\right] .
\end{equation}

In our setting, we consider a new type of recursive utilities which is a
mean-field delayed BSDE (MFDBSDE) with jumps and the time horizon is
possibly infinite, then%
\begin{equation}
\begin{array}{c}
dY(t)=-f(t,Y_{t},Z_{t},K_{t}(\cdot ),\mathbb{E}[Y(t)],\mathbb{E}[Z(t)],%
\mathbb{E}[K(t,\cdot )],\pi (t))dt+Z(t)dB(t) \\
+\int_{%
\mathbb{R}
_{0}}K(t,\zeta )\tilde{N}(dt,d\zeta ),t\in \left[ 0,\tau \right] \text{ for }%
\tau \leq \infty .%
\end{array}
\label{1st}
\end{equation}

Here
\begin{equation*}
\begin{array}{l}
Y_{t}=Y(t+s), \\
Z_{t}=Z(t+s), \\
K_{t}(\cdot )=K(t+s,\cdot ),%
\end{array}%
\end{equation*}%
for a given $s\in \left[ -\delta ,0\right] .$

Now, for any stochastic process $X$ is a solution of a stochastic functional
differential equation, let the bold $\mathbf{X}$ denotes
\begin{equation*}
\mathbf{X}(t):=\Big(\int_{-\delta }^{0}X(t+s)\mu _{1}(ds),\ldots
,\int_{-\delta }^{0}X(t+s)\mu _{N}(ds)\Big),
\end{equation*}%
where each $\mu _{i}$ is a bounded Borel measure, which is either absolutely
continuous w.r.t. the Lebesgue measure or a Dirac measure $\delta _{t_{0}}$
for some $t_{0}\in \lbrack -\delta ,0].$ Dealing with the maximum principle,
the authors in the paper \cite{AR}, consider a control problem involving
decoupled system of finite and infinite time horizon mean-field forward
backward stochastic differential equations (MFFBSDE). The backward part of
these equations in \cite{AR}, has the form (together with either a terminal
condition or a decay condition, depending on the given time horizon):

\textit{The state equation}%
\begin{equation}
\begin{array}{c}
dY(t)=-f(t,\mathbf{X}(t),Y(t),Z(t),K(t,\cdot ),\mathbb{E}[\mathbf{X}(t)],%
\mathbb{E}[Y(t)],\pi (t))dt+Z(t)dB(t) \\
+\int_{%
\mathbb{R}
_{0}}K(t,\zeta )\tilde{N}(dt,d\zeta ),t\in \lbrack 0,\infty ),%
\end{array}
\label{EQ1}
\end{equation}

\textit{The corresponding derivative process}

{\normalsize
\begin{equation}
\begin{split}
d\mathcal{Y}(t)& =\Big(-\nabla g(t,\pi )\Big)\cdot \Big(\mathcal{X}(t),%
\mathcal{Y}(t),\mathbb{E}[\mathcal{X}(t)],\mathbb{E}[\mathcal{Y}(t)],%
\mathcal{Z}(t),\mathcal{K}(t,\cdot ),\eta (t)\Big)^{\mathsf{T}}dt \\
& +\mathcal{Z}(t)dB(t)+\int_{%
\mathbb{R}
_{0}}\mathcal{K}(t,\zeta )\tilde{N}(dt,d\zeta ),t\in \lbrack 0,\infty ),
\end{split}
\label{eq:DerivativeYZK}
\end{equation}%
}where $X$ and $\mathcal{X}$ are solutions of forward equations that does
not depend on $Y,Z,K$ and $\mathcal{Y},\mathcal{Z},\mathcal{K},$
respectively and we denote by $\left( \cdot \right) ^{T}$ the transpose.

In this note, we establish sufficient conditions for existence and
uniqueness for such equations, in both finite and infinite time horizon.

\paragraph{Finite time horizon}

In section $2,$ we consider different types of delayed backward stochastic
differential equation (DBSDE). In the mean-field delay equation $\left( \ref%
{1st}\right) ,$ we allow the coefficient functionals to depend on the
segments of the process in a similar manner as in the SFDEs in \cite{SEM},
\cite{AR}, and the Backward stochastic functional differential equations
(BSFDEs) in \cite{CP}. Also related to our equation, are the BSDEs with time
delayed generator in \cite{DI2}. The authors in \cite{DI2} consider
equations similar to ours, only with the Poisson random measures from $%
\left( \ref{1st}\right) $ replaced by the random measure $\tilde{M}%
(dt,d\zeta ):=\zeta \tilde{N}(dt,d\zeta ).$

\paragraph{Infinite time horizon}

In section $3,$ we study the infinite horizon case of the MFDBSDE $\left( %
\ref{1st}\right) .$ It can be seen as an extension of Theorem $3.1$ in \cite%
{HOP} to mean-field with delay.

\section{Finite time horizon}

Let $B(t),t\geq 0$ be a Brownian motion, and $\tilde{N}(dt,d\zeta ),t\geq 0$
be an independent compensated Poisson random measure, with compensator $\nu
(\zeta )dt,$ on a probability space $(\Omega ,\mathcal{F},P).$ Let $(%
\mathcal{F}_{t})_{t\geq 0}$ denote the natural filtration associated with $B$
and $N.$ Let $\delta >0,$ and extend the filtration by letting $\mathcal{F}%
_{t}=\mathcal{F}_{0}$ for $t\in \lbrack -\delta ,0].$

Consider an equation of the form

\begin{equation}
\begin{array}{l}
dY(t)=-f(t,{Y}_{t},{Z}_{t},{K}_{t}(\cdot ),\mathbb{E}[Y(t)],\mathbb{E}[Z(t)],%
\mathbb{E}[K(t,\cdot )])dt \\
+Z(t)dB(t)+\int_{\mathbb{R}_{0}}K(t,\zeta )\tilde{N}(dt,d\zeta );t\in
\lbrack 0,T], \\
Y(T)=\xi ,%
\end{array}
\label{eq:BSFDE}
\end{equation}%
where
\begin{align*}
{Y}_{t}(s,\omega ):& =%
\begin{cases}
Y(t+s,\omega ), & s\in \lbrack -\delta ,0],t\geq 0,\omega \in \Omega , \\
Y(0) & t<0%
\end{cases}
\\
{Z}_{t}(s,\omega ):& =%
\begin{cases}
Z(t+s,\omega ), & s\in \lbrack -\delta ,0],t\geq 0,\omega \in \Omega , \\
0 & t<0%
\end{cases}
\\
{K}_{t}(s,\omega )(\zeta ):& =%
\begin{cases}
K(t+s,\omega ,\zeta ), & s\in \lbrack -\delta ,0],t\geq 0,\omega \in \Omega
,\zeta \in \mathbb{R}_{0}, \\
0 & t<0%
\end{cases}%
\end{align*}%
and $\xi \in L^{2}(\Omega ,\mathcal{F}_{T}).$

Here, for each $t,$ $(Y_{t},Z_{t},K_{t})$ are assumed to belong to the space
\begin{equation*}
\mathbb{S}^{2}\times \mathbb{L}^{2}\times \mathbb{H}^{2},
\end{equation*}%
of functionals defined below.

\begin{enumerate}
\item[i)] $\mathbb{S}^{2}=\mathbb{S}^{2}(\Omega ,\mathcal{D}[-\delta ,0])$
consisting of the functions $\alpha :\Omega \rightarrow \mathcal{D}([-\delta
,0],\mathbb{R})$ such that $\omega \mapsto \alpha (\omega ,s)$ is $\mathcal{F%
}_{t+s}$-measurable for each $s\in \lbrack -\delta ,0],$ and $s\mapsto
\alpha (\omega ,s)$ is c\`{a}dl\`{a}g for each $\omega \in \Omega .$ Let $%
\mathbb{S}^{2}$ be equipped with the norm
\begin{equation*}
\parallel \alpha \parallel _{\mathbb{S}^{2}}^{2}:=\mathbb{E}\left[
\sup_{s\in \lbrack -\delta ,0]}|\alpha (s)|^{2}\right] .
\end{equation*}
\end{enumerate}

We refer to \cite{BCDDR} for more on this space in connection with \emph{%
stochastic functional differential equations}.

\begin{enumerate}
\item[ii)] $\mathbb{L}^{2}$ consisting of the functions
\begin{equation*}
\lambda :\Omega \times \lbrack -\delta ,0]\rightarrow \mathbb{R}
\end{equation*}%
such that $(s,\omega )\mapsto \lambda (\omega ,s)$ is measurable with
respect to the predictable $\sigma $-algebra $\mathcal{P},$ generated by the
filtration $\mathcal{F}_{t+s},s\in \lbrack -\delta ,0].$ Now, we equip $%
\mathbb{L}^{2}$ with the norm
\begin{equation*}
\parallel \lambda \parallel _{\mathbb{L}^{2}}^{2}:=\mathbb{E}\left[
\int_{-\delta }^{0}\left\vert \lambda (t)\right\vert ^{2}dt\right] <\infty .
\end{equation*}
\end{enumerate}

Now, let $\mu $ be some bounded Borel measure on $[-\delta ,0]$ which is
either the Lebesgue measure or a Dirac point measure.

\begin{enumerate}
\item[iii)] $\mathbb{H}^{2}:=\mathbb{H}^{2}(\Omega ;L^{2}(\mu ))$ consisting
of the functions
\begin{equation*}
\theta :\Omega \times \lbrack -\delta ,0]\times \mathbb{R}_{0}\rightarrow
\mathbb{R}
\end{equation*}%
such that $(s,\omega ,\zeta )\mapsto \theta (\omega ,s,\zeta )$ is
measurable with respect to the product $\mathcal{P}\otimes B(\mathbb{R}_{0}),
$ generated by the filtration $\{\mathcal{F}_{t+s}\}_{s\in \lbrack -\delta
,0]},$ while $B$ is a Borel set of $\mathbb{R}_{0}.$ Now, we equip $\mathbb{H%
}^{2}$ with the norm
\begin{equation*}
\parallel \theta \parallel _{\mathbb{H}^{2}}^{2}:=\mathbb{E}\left[ \int_{%
\mathbb{R}_{0}}\int_{-\delta }^{0}\left\vert \theta (t,\zeta )\right\vert
^{2}\nu (d\zeta )dt\right] <\infty .
\end{equation*}
\end{enumerate}

Also, define the spaces

\begin{itemize}
\item $\mathbf{S}^{2}$ consisting of the c\`{a}dl\`{a}g adapted processes
\begin{equation*}
P:\Omega \times \lbrack -\delta ,T]\rightarrow \mathbb{R}
\end{equation*}%
such that $P(s)=P(0)$ for $s\in \lbrack -\delta ,0)$ and $\mathbb{E}\left[
\underset{t\in \left[ 0,T\right] }{\sup }|P(t)|^{2}\right] <\infty .$ We
equip $\mathbf{S}^{2}$ with equivalent norms
\begin{equation*}
\parallel P\parallel _{\mathbf{S}_{\beta }^{2}}^{2}:=\mathbb{E}\left[
\sup_{t\in \lbrack 0,T]}e^{\beta t}|P(t)|^{2}\right] ,\beta >0.
\end{equation*}

\item $\mathbf{L}^{2}$ consisting of the predictable processes
\begin{equation*}
Q:\Omega \times \lbrack -\delta ,T]\rightarrow \mathbb{R}
\end{equation*}%
such that $Q(s)=0$ for each $s\in \lbrack -\delta ,0),$ $%
\int_{0}^{T}Q(t)^{2}dt<\infty $. We equip $\mathbf{L}^{2}$ with the
equivalent norms
\begin{equation*}
\parallel Q\parallel _{\mathbf{L}_{\beta }^{2}}^{2}:=\int_{0}^{T}e^{\beta
t}Q(t)^{2}dt,\beta >0.
\end{equation*}

\item $\mathbf{H}^{2}$ consisting of the predictable processes
\begin{equation*}
R:\Omega \times \lbrack -\delta ,T]\times \mathbb{R}_{0}\rightarrow \mathbb{R%
}
\end{equation*}%
such that $R(s,\zeta )=0$ for each $s\in \lbrack -\delta ,0)$ and $%
\int_{0}^{T}\int_{\mathbb{R}_{0}}R(u,\zeta )^{2}\nu (d\zeta )dt<\infty $ We
equip $\mathbf{H}^{2}$ with the equivalent norms
\begin{equation*}
\parallel R\parallel _{\mathbf{H}_{\beta }^{2}}^{2}:=\int_{0}^{T}\int_{%
\mathbb{R}_{0}}e^{\beta t}R(t,\zeta )^{2}\nu (d\zeta )dt,\beta >0.
\end{equation*}
\end{itemize}

\begin{definition}
We say that $(Y,Z,K)\in \mathbf{S}^{2}\times \mathbf{L}^{2}\times \mathbf{H}%
^{2}$ is a \emph{solution} to $\left( \ref{eq:BSFDE}\right) ,$ if
\begin{equation}
\begin{array}{l}
Y(t)=\xi +\int_{t}^{T}f(s,Y_{s},Z_{s},K_{s},\mathbb{E}[Y(s)],\mathbb{E}%
[Z(s)],\mathbb{E}[K(s,\cdot )])ds \\
-\int_{t}^{T}Z(s)dB(s)-\int_{t}^{T}\int_{\mathbb{R}_{0}}K(s,\zeta )\tilde{N}%
(ds,d\zeta );t\in \lbrack 0,T].%
\end{array}
\label{simple}
\end{equation}
\end{definition}

{\normalsize \textbf{Assumption (H1)}}

{\normalsize Let $f:\Omega \times \lbrack 0,T]\times \mathbb{S}^{2}\times
\mathbb{L}^{2}\times \mathbb{H}^{2}\rightarrow
\mathbb{R}
,$ an }$\mathcal{F}_{t}$-adapted{\normalsize \ and $\xi $ satisfy the
following assumptions: }

\begin{enumerate}
\item[i)]  $\xi \in L^{2}(\Omega ,\mathcal{F}_{T})$.
\item[ii)] There exists a constant $C_{f}>0$ such that
\begin{equation*}
\begin{array}{l}
\left\vert f(t,y_{1},z_{1},k_{1}(\cdot ),y_{2},z_{2},k_{2}(\cdot ))-f(t,%
\tilde{y}_{1},\tilde{z}_{1},\tilde{k}_{1}(\cdot ),\tilde{y}_{2},\tilde{z}%
_{2},\tilde{k}_{2}(\cdot ))\right\vert ^{2} \\
\leq C_{f}\left( \left\vert y_{1}-\tilde{y}_{1}\right\vert ^{2}+\left\vert
z_{1}-\tilde{z}_{1}\right\vert ^{2}+\int_{
\mathbb{R}
_{0}}\left\vert k_{1}\left( \zeta \right) -\tilde{k}_{1}\left( \zeta \right)
\right\vert ^{2}\nu (d\zeta )\right. \\
\left. +\left\vert y_{2}-\tilde{y}_{2}\right\vert +\left\vert z_{2}-\tilde{z}%
_{2}\right\vert +\int_{
\mathbb{R}
_{0}}\left\vert k_{2}\left( \zeta \right) -\tilde{k}_{2}\left( \zeta \right)
\right\vert \nu (d\zeta )\right) \\
\text{ {\normalsize a.e.} }t,y_{1},\tilde{y}_{1},y_{2},\tilde{y}_{2},z_{1},%
\tilde{z}_{1},k_{1}(\cdot ),\tilde{k}_{1}(\cdot ),k_{2}(\cdot ),\tilde{k}%
_{2}(\cdot ).%
\end{array}%
\end{equation*}

\item[iii)]
\begin{equation*}
\mathbb{E}\left[ \int_{0}^{T}\left\vert f(t,0,0,0,0,0,0)\right\vert ^{2}dt%
\right] <\infty .
\end{equation*}
\end{enumerate}

\begin{lemma}
{\normalsize {\small Suppose that $(Y,f,Z,K,\xi )$}}${\normalsize {\small ,}}
${\normalsize {\small \ $(\tilde{Y},\tilde{f},\tilde{Z},\tilde{K},\tilde{\xi}%
)$ are arbitrary processes in }}$%
\begin{array}{c}
{\normalsize {\small \mathbf{S}^{2}\times \mathbf{%
\mathbb{R}
}\times \mathbf{L}^{2}\times \mathbf{H}^{2}\times L^{2}(\mathcal{F}_{T}),}}%
\end{array}%
${\normalsize {\small \ that satisfies the equation }}${\normalsize {\small (%
\ref{eq:BSFDE}).}}${\normalsize {\small \ Then under Assumption (H1),\ there
exists a constant $C_{\beta }>0,$ such that
\begin{align*}
& \parallel Y-\tilde{Y}\parallel _{\mathbf{S}_{\beta }^{2}}^{2}+\parallel Z-%
\tilde{Z}\parallel _{\mathbf{L}_{\beta }^{2}}^{2}+\parallel K-\tilde{K}%
\parallel _{\mathbf{H}_{\beta }^{2}}^{2} \\
& \leq C_{\beta }\left( \mathbb{E}\left[ \left\vert \xi -\tilde{\xi}%
\right\vert ^{2}\right] +\mathbb{E}\left[ \int_{0}^{T}e^{\beta t}\left\vert
f(t)-\tilde{f}(t)\right\vert ^{2}dt\right] \right) .
\end{align*}%
}}

{\normalsize {\small For convenience, we have used the simplified notation }}
\begin{eqnarray*}
f(t) &=&f(t,{Y}_{t},{Z}_{t},{K}_{t}(\cdot ),\mathbb{E}[Y(t)],\mathbb{E}%
[Z(t)],\mathbb{E}[K(t,\cdot )]) \\
\tilde{f}(t) &=&f(t,{\tilde{Y}}_{t},{\tilde{Z}}_{t},{\tilde{K}}_{t}(\cdot ),%
\mathbb{E}[\tilde{Y}(t)],\mathbb{E}[\tilde{Z}(t)],\mathbb{E}[\tilde{K}%
(t,\cdot )])
\end{eqnarray*}

The constant $C_{\beta }$ can be chosen, such that{\normalsize \ }%
\begin{equation*}
C_{\beta }:=\max \left( 9e^{\beta T},8T+\frac{1}{\beta }\right) ,\text{ \ \
\ \ }\beta >0.
\end{equation*}
\end{lemma}

\begin{proof}
{\normalsize We proceed as in \cite{DI2}. Introduce the notations $\Delta
Y=Y-\tilde{Y}$}${\normalsize ,}${\normalsize \ $\Delta Z=Z-\tilde{Z}$ and so
on. By applying It\^{o}'s formula to $e^{\beta t}|\Delta Y(t)|$}$%
{\normalsize ^{2},}${\normalsize \ we find that
\begin{align*}
de^{\beta t}|\Delta Y(t)|^{2}& =\Big[\beta e^{\beta t}|\Delta
Y(t)|^{2}+e^{\beta t}|\Delta Z(t)|^{2} \\
& +e^{\beta t}\int_{\mathbb{R}_{0}}|\Delta K(t,\zeta )|^{2}\nu (d\zeta
)-2e^{\beta t}\Delta Y(t)\Delta f(t)\Big]dt \\
& +2e^{\beta t}\Delta Y(t)\Delta Z(t)dB(t) \\
& +\int_{\mathbb{R}_{0}}e^{\beta t}\left( |\Delta K(t,\zeta )|^{2}+2e^{\beta
t}\Delta Y(t)\Delta K(t,\zeta )\right) \tilde{N}(dt,d\zeta ).
\end{align*}%
We claim now that, the $dB(t)$ and $\tilde{N}(dt,d\zeta )$-intergals above
have mean zero. To see why this holds, consider the processes {\small
\begin{align}
M_{1}(t)& :=\int_{0}^{t}2e^{\beta u}\Delta Y(u)\Delta Z(u)dB(u) \\
M_{2}(t)& :=\int_{0}^{t}e^{\beta u}\left( |\Delta K(u,\zeta )|^{2}+e^{\beta
u}\Delta Y(u)\Delta K(u,\zeta )\right) \tilde{N}(du,d\zeta )
\end{align}%
separately. Now since the integrands $|\Delta K(u,\zeta )|^{2}+e^{\beta
u}\Delta Y(u)\Delta K(u)$ is $P\otimes dt\otimes \nu $-integrable, $M_{2}(t)$
have mean zero (in fact $M_{2}$ is a martingale, see e.g. \cite{MR2460554},%
\cite{MR2603057}). }}

In order to find an explicit expression for $C_{\beta },$ we proceed as in
Lemma $2.1$ in {\normalsize \cite{DI2}}.
\end{proof}

\begin{theorem}
\label{existence} {\normalsize Suppose that Assumption (H1)\ holds, and that
there is some $\beta \geq 0$ such that
\begin{equation}
C_{\beta }C_{f}(1+e^{-\beta s})\max (1,T)<1.  \label{**}
\end{equation}%
}

{\normalsize Then there exists a unique solution $(Y,Z,K)\in \mathbf{S}%
^{2}\times \mathbf{L}^{2}\times \mathbf{H}^{2}$ to equation (\ref{eq:BSFDE}%
). }
\end{theorem}

\begin{proof}
{\normalsize It is a classical result of existence and uniqueness of
solutions for BSDEs with jumps, proved for $\beta =0$ in \cite{BBP}, Theorem
}${\normalsize 2.1}${\normalsize \ and Proposition }${\normalsize 2.2.}$%
{\normalsize \ We also refer to \cite{DI2}, \cite{TL} and \cite{R} for
related results, and to \cite{B2006} for generalizations using conditional
expectation with respect to stopping times and with non-homogeneous
compensators replacing $\nu \otimes dt$}${\normalsize .}${\normalsize \ Now
since the $\beta $-norms are equivalent for every $\beta $}${\normalsize ,}$%
{\normalsize \ the results holds for arbitrary $\beta $}${\normalsize .}$

{\normalsize We show the existence of a solution following \cite{DI2}.
Define the maps }$%
\begin{array}{c}
\Upsilon \left( Y,Z,K\right) :=\left( \bar{Y},\bar{Z},\bar{K}\right) ,%
\end{array}%
$ for\ $\mathbf{S}_{\beta }^{2}\times \mathbf{L}_{\beta }^{2}\times \mathbf{H%
}_{\beta }^{2}\rightarrow \mathbf{S}_{\beta }^{2}\times \mathbf{L}_{\beta
}^{2}\times \mathbf{H}_{\beta }^{2},$ {\normalsize by
\begin{align*}
d\bar{Y}(t)& =-f(t,Y_{t},Z_{t},K_{t}(\cdot ),\mathbb{E}[Y(t)],\mathbb{E}%
[Z(t)],\mathbb{E}[K(t,\cdot )])dt & & \\
& +\bar{Z}(t)dB(t)+\int_{\mathbb{R}_{0}}\bar{K}(t,\zeta )\tilde{N}(dt,d\zeta
), & 0& \leq t\leq T, \\
\bar{Y}(T)& =\xi  & & \\
\bar{Y}(t)& =\bar{Y}(0)\text{ and }\bar{Z}(t)=\bar{K}(t,\cdot )=0 & t& \in
\lbrack -\delta ,0].
\end{align*}%
By assumption, the integrands }$f(t,Y_{t},Z_{t},K_{t}(\cdot ),\mathbb{E}%
[Y(t)],\mathbb{E}[Z(t)],\mathbb{E}[K(t,\cdot )])${\normalsize \ has a
version in $\mathbf{L}^{2}(P\otimes \mu )$ whenever $Y,Z,K$ satisfy the
Lipschitz condition }${\normalsize (ii)}$ of Assumption (H1){\normalsize ,\
and hence by Lemma }${\normalsize 2,}${\normalsize \ the processes $(\bar{Z}$%
}${\normalsize ,}${\normalsize $\bar{Y}$}${\normalsize ,}${\normalsize $\bar{%
K})$ are well defined. Now, it suffices to show that }$\left( \bar{Y},\bar{Z}%
,\bar{K}\right) =:${\normalsize $\Upsilon $}$\left( Y,Z,K\right) $%
{\normalsize \ is a contraction for every $\beta \geq 0$ such that }$%
{\normalsize (\ref{**})}$ is satisfied{\normalsize . For that, we consider }$%
\begin{array}{c}
\left( Y^{1},Z^{1},K^{1}\right) ,\left( Y^{2},Z^{2},K^{2}\right) \in \mathbf{%
S}_{\beta }^{2}\times \mathbf{L}_{\beta }^{2}\times \mathbf{H}_{\beta }^{2}%
\end{array}%
$ with
\begin{equation*}
\begin{array}{c}
\left( \bar{Y}^{1},\bar{Z}^{1},\bar{K}^{1}\right) =:{\normalsize \Upsilon }%
\left( Y^{1},Z^{1},K^{1}\right) , \\
\left( \bar{Y}^{2},\bar{Z}^{2},\bar{K}^{2}\right) =:{\normalsize \Upsilon }%
\left( Y^{2},Z^{2},K^{2}\right) .%
\end{array}%
\end{equation*}

 We proceed by showing some useful inequalities. Suppose that $\mathbf{t}\in \left[ 0,T\right] ,$ $P\in \mathbf{S}^{2},Q\in \mathbf{L}^{2}$
and $R\in \mathbf{H}^{2}.$
We have for $s\in \left[ -\delta ,0\right] :$
\begin{equation}
\begin{array}{ll}
\mathbb{E}\left[ \int_{0}^{T}e^{\beta u}|P(u+s)|^{2}du\right] & =e^{-\beta s}%
\mathbb{E}\left[ \int_{0}^{T}e^{\beta (u+s)}|P(u+s)|^{2}du\right] \\
& \leq e^{-\beta s}\mathbb{E}\left[ \int_{0}^{T}\underset{0\leq u\leq T}{%
\sup }e^{\beta (u+s)}|P(u+s)|^{2}du\right] \\
& =Te^{-\beta s}\parallel P\parallel _{\mathbf{S}_{\beta }^{2}}^{2},%
\end{array}
\label{n1}
\end{equation}%
and%
\begin{equation}
\begin{array}{ll}
\mathbb{E}\left[ \int_{0}^{T}e^{\beta u}\left( \mathbb{E[}\left\vert
P(u)\right\vert ]\right) ^{2}du\right] & \leq \mathbb{E}\left[ \int_{0}^{T}%
\underset{0\leq u\leq T}{\sup }e^{\beta u}\left\vert P(u)\right\vert ^{2}du%
\right] \\
& =T\parallel P\parallel _{\mathbf{S}_{\beta }^{2}}^{2}.%
\end{array}
\label{n3}
\end{equation}

Choosing $v:=u+s,$ we get%
\begin{equation}
\begin{array}{ll}
\mathbb{E}\left[ \int_{0}^{T}e^{\beta u}|{\normalsize Q}(u+s)|^{2}du\right]
& =\mathbb{E}\left[ \int_{s}^{T+s}e^{\beta (v-s)}|{\normalsize Q}(v)|^{2}dv%
\right] \\
& =e^{-\beta s}\mathbb{E}\left[ \int_{s}^{T+s}e^{\beta v}|{\normalsize Q}%
(v)|^{2}dv\right] \\
& \leq e^{-\beta s}\mathbb{E}\left[ \int_{0}^{T}e^{\beta v}|{\normalsize Q}%
(v)|^{2}dv\right] \\
& =e^{-\beta s}\parallel {\normalsize Q}\parallel _{\mathbf{L}_{\beta
}^{2}}^{2},%
\end{array}
\label{n2}
\end{equation}%
and%
\begin{eqnarray}
\mathbb{E}\left[ \int_{0}^{T}e^{\beta u}\left( \mathbb{E[}\left\vert
{\normalsize Q}(u)\right\vert ]\right) ^{2}du\right] &\leq &\mathbb{E}\left[
\int_{0}^{T}e^{\beta u}\left\vert {\normalsize Q}(u)\right\vert ^{2}du%
\right]  \label{l1} \\
&=&\parallel {\normalsize Q}\parallel _{\mathbf{L}_{\beta }^{2}}^{2}.  \notag
\end{eqnarray}

Similar estimates $\left( \ref{n2}\right) -\left( \ref{l1}\right) ,$ for
{\normalsize $R\in \mathbf{H}^{2}.$\ }

{\normalsize Now, by Lemma }${\normalsize 2,}${\normalsize \ the Lipschitz
Assumption and by the inequalities }$\left( \ref{n1}\right) -\left( \ref{l1}%
\right) {\normalsize ,}${\normalsize \ we find that
\begin{align*}
& \parallel \bar{Y}^{1}-\bar{Y}^{2}\parallel _{\mathbf{S}_{\beta
}^{2}}^{2}+\parallel \bar{Z}^{1}-\bar{Z}^{2}\parallel _{\mathbf{L}_{\beta
}^{2}}^{2}+\parallel \bar{K}^{1}-\bar{K}^{2}\parallel _{\mathbf{H}_{\beta
}^{2}}^{2} \\
& \leq C_{\beta }C_{f}\left( 1+e^{-\beta s}\right) \max (1,T)\Big(\parallel
Y^{1}-Y^{2}\parallel _{\mathbf{S}_{\beta }^{2}}^{2}+\parallel
Z^{1}-Z^{2}\parallel _{\mathbf{L}_{\beta }^{2}}^{2}+\parallel
K^{1}-K^{2}\parallel _{\mathbf{H}_{\beta }^{2}}^{2}\Big).
\end{align*}%
}

Whenever $\left( \ref{**}\right) $ is satisfied, the proof is completed.
\end{proof}

\begin{remark}
{\normalsize If the Lipschitz condition }${\normalsize (ii)}${\normalsize \
is replaced by stronger Lipschitz condition }${\normalsize (A3)}$%
{\normalsize \ in \cite{DI2}, we may choose }%
\begin{equation*}
C_{\beta }C_{f}\left( 2+\int_{-\delta }^{0}e^{-\beta s}\mu (ds)\right) \max
(1,T)<1,
\end{equation*}%
instead of inequality $\left( \ref{**}\right) $ in Theorem \ref{existence}
above{\normalsize .}
\end{remark}

\subsection{Special case}

We consider now a mean-field equation on the form%
\begin{equation}
\begin{array}{c}
Y(t)=\xi +\int_{t}^{T}f(u,\mathbf{Y}(u),\mathbf{Z}(u),\mathbf{K}(u,\cdot ),%
\mathbb{E}[Y(u)],\mathbb{E}[Z(u)],\mathbb{E}[K(u,\cdot )])du \\
+\int_{t}^{T}Z(u)dB(u)+\int_{t}^{T}\int_{%
\mathbb{R}
_{0}}K(u,\zeta )\tilde{N}(du,d\zeta );t\in \left[ 0,T\right] ,%
\end{array}
\label{lem}
\end{equation}%
where for a given positive constant $\delta .$

The boldface processes from equation $\left( \ref{lem}\right) $ are defined
by \ \
\begin{equation*}
\begin{array}{l}
\mathbf{Y}(t):=(Y(t),Y(t-\delta )), \\
\mathbf{Z}(t):=(Z(t),Z(t-\delta )), \\
\mathbf{K}(t,\cdot ):=(K(t,\cdot ),K(t-\delta ,\cdot )),%
\end{array}%
\end{equation*}%
and a generator $f:\Omega \times \left[ 0,T\right] \times
\mathbb{R}
\times \mathbb{S}^{2}\times
\mathbb{R}
\times \mathbb{L}^{2}\times L^{2}(\nu )\times \mathbb{H}^{2}\times
\mathbb{R}
^{3}\rightarrow \times
\mathbb{R}
.$

For $f$ and $\xi ,$ we impose the set of assumptions.

{\normalsize \textbf{Assumption (H2)}}

\begin{enumerate}
\item[i)] {\normalsize $\xi \in L^{2}(\Omega ,\mathcal{F}_{T})$}$%
{\normalsize .}${\normalsize \ }

\item[ii)] The generator function $f$ is $\mathcal{F}_{t}$-adapted and
Lipschitz in the sense that {\normalsize there exists a constant $\hat{C}%
_{f}>0$ such that}%
\begin{equation*}
\begin{array}{l}
\left\vert f(t,y_{1},y_{2},z_{1},z_{2},k_{1}(\cdot ),k_{2}(\cdot
),y_{3},z_{3},k_{3}(\cdot ))-f(t,\tilde{y}_{1},\tilde{y}_{2},\tilde{z}_{1},%
\tilde{z}_{2},\tilde{k}_{1}(\cdot ),\tilde{k}_{2}(\cdot ),\tilde{y}_{3},%
\tilde{z}_{3},\tilde{k}_{3}(\cdot ))\right\vert ^{2} \\
\leq \hat{C}_{f}\left( \left\vert y_{1}-\tilde{y}_{1}\right\vert
^{2}+\left\vert y_{2}-\tilde{y}_{2}\right\vert ^{2}+\left\vert z_{1}-\tilde{z%
}_{1}\right\vert ^{2}+\left\vert z_{2}-\tilde{z}_{2}\right\vert ^{2}\right.
\\
+\int_{%
\mathbb{R}
_{0}}\left\vert k_{1}\left( \zeta \right) -\tilde{k}_{1}\left( \zeta \right)
\right\vert ^{2}\nu (d\zeta )+\int_{%
\mathbb{R}
_{0}}\left\vert k_{2}\left( \zeta \right) -\tilde{k}_{2}\left( \zeta \right)
\right\vert ^{2}\nu (d\zeta ) \\
\left. +\left\vert y_{3}-\tilde{y}_{3}\right\vert +\left\vert z_{3}-\tilde{z}%
_{3}\right\vert +\int_{%
\mathbb{R}
_{0}}\left\vert k_{3}\left( \zeta \right) -\tilde{k}_{3}\left( \zeta \right)
\right\vert \nu (d\zeta )\right) \\
\text{ {\normalsize a.e.} }t,y_{1},\tilde{y}_{1},y_{2},\tilde{y}_{2},y_{3},%
\tilde{y}_{3},z_{1},\tilde{z}_{1},z_{2},\tilde{z}_{2},z_{3},\tilde{z}%
_{3},k_{1}(\cdot ),\tilde{k}_{1}(\cdot ),k_{2}(\cdot ),\tilde{k}_{2}(\cdot
),k_{3}(\cdot ),\tilde{k}_{3}(\cdot ).%
\end{array}%
\end{equation*}

\item[iii)]
\begin{equation*}
\mathbb{E}\left[ \int_{0}^{T}\left\vert f(t,0,0,0,0,0,0,0,0,0)\right\vert
^{2}dt\right] <\infty .
\end{equation*}
\end{enumerate}

\begin{theorem}
We say that $\left( Y,Z,K\right) \in \mathbf{S}^{2}\times \mathbf{L}%
^{2}\times \mathbf{H}^{2}$ is a unique solution of $\left( \ref{lem}\right) $
if $\xi $ and $f$ satisfy the Assumption (H2), and if there is $\beta \geq
0, $ such that
\begin{equation*}
C_{\beta }\hat{C}_{f}(2+e^{\beta \delta })\max (1,T)<1.
\end{equation*}
\end{theorem}

\begin{proof}
The proof is closely related to the proof of Theorem $4.$ We just need some
estimations instead of estimations $\left( \ref{n1}\right) $ and $\left( \ref%
{n2}\right) .$

 Suppose that $\mathbf{t}\in \left[ 0,T\right] ,$ $P\in \mathbf{S}^{2},Q\in \mathbf{L}^{2}$
and $R\in \mathbf{H}^{2}.$ We have that%
\begin{equation*}
\begin{array}{ll}
\mathbb{E}\left[ \int_{0}^{T}e^{\beta u}|P(u-\delta )|^{2}du\right] &
=e^{\beta \delta }\mathbb{E}\left[ \int_{0}^{T}e^{\beta (u-\delta
)}|P(u-\delta )|^{2}du\right] \\
& \leq e^{\beta \delta }\mathbb{E}\left[ \int_{0}^{T}\underset{0\leq u\leq T%
}{\sup }e^{\beta (u-\delta )}|P(u-\delta )|^{2}du\right] \\
& =Te^{\beta \delta }\parallel P\parallel _{\mathbf{S}_{\beta }^{2}}^{2},%
\end{array}%
\end{equation*}%
and that%
\begin{equation*}
\begin{array}{ll}
\mathbb{E}\left[ \int_{0}^{T}e^{\beta u}|{\normalsize Q}(u-\delta )|^{2}du%
\right] & =\mathbb{E}\left[ \int_{-\delta }^{T-\delta }e^{\beta (v+\delta )}|%
{\normalsize Q}(v)|^{2}dv\right] \\
& \leq e^{\beta \delta }\parallel {\normalsize Q}\parallel _{\mathbf{L}%
_{\beta }^{2}}^{2},%
\end{array}%
\end{equation*}%
similarly for {\normalsize $R\in \mathbf{H}^{2}.$}
\end{proof}

\begin{remark}
Related to the mean field BSDEs with delay $\left( 2\right) ,$ $\left(
5\right) $ and $\left( \ref{boldy}\right) $ are the BSDEs $\left( 3\right) .$
This fully coupled BSDEs allows for the delay generator by the FSDEs. We can
solve the forward equation separately to obtain the process $X$ and then
plugging the solution $X$ into the backward equation to find the solution $%
\left( Y,Z,K\right) .$
\end{remark}

\section{Infinite horizon case}

We deal now, with the existence and uniqueness of a solution $\left(
Y,Z,K\right) $ of the following MFDBSDE, in infinite time horizon:

\begin{equation}
\left\{
\begin{array}{l}
dY(t)=-f(t,Y_{t},Z_{t},K_{t}(\cdot ),\mathbb{E}[Y(t)],\mathbb{E}[Z(t)],%
\mathbb{E}[K(t,\cdot )])dt+Z(t)dB(t) \\
+\int_{%
\mathbb{R}
_{0}}K(t,\zeta )\tilde{N}\left( dt,d\zeta \right) ;t\geq 0, \\
\underset{t\rightarrow \infty }{\lim }Y(t)=0.%
\end{array}%
\right.  \label{1.1}
\end{equation}

Here the generator $f:\Omega \times \lbrack 0,\infty )\times \mathbb{S}%
^{2}\times \mathbb{L}^{2}\times \mathbb{H}^{2}\times
\mathbb{R}
^{3}\rightarrow
\mathbb{R}
,$ depends on previous values of the solution, such as

\begin{equation*}
\begin{array}{l}
Y_{t}:=Y(t+r), \\
Z_{t}:=Z(t+r), \\
K_{t}(\cdot ):=K(t+r,\cdot ),%
\end{array}%
\end{equation*}%
for some constant $r\in \left[ -\delta ,0\right] .$

\begin{definition}
We denote by $\mathcal{L},$ the space of all $\mathcal{F}_{t}$-adapted
processes satisfying:%
\begin{equation}
\left\Vert \left( Y,Z,K\right) \right\Vert ^{2}:=\mathbb{E}\left[ \underset{%
t\geq 0}{\sup }e^{\beta t}\left\vert Y(t)\right\vert ^{2}+\int_{0}^{\infty
}e^{\beta t}\left\{ \left\vert Z(t)\right\vert ^{2}+\int_{%
\mathbb{R}
_{0}}\left\vert K(t,\zeta )\right\vert ^{2}\nu \left( d\zeta \right)
\right\} dt\right] <\infty .  \label{1.2}
\end{equation}
\end{definition}

We set now assumptions proving existence and uniqueness of the solution of
equation $\left( \ref{1.1}\right) .$

\textbf{Assumption (H3)}

\textit{On the generator }$f:$

\begin{itemize}
\item The function $f$ is $\mathcal{F}_{t}$-adapted.

\item \textit{Integrability condition:}%
\begin{equation*}
\mathbb{E}\left[ \int_{0}^{\infty }e^{\beta t}\left\vert
f(t,0,0,0,0,0,0)\right\vert ^{2}dt\right] <\infty .
\end{equation*}

\item \textit{Lipschitz condition: }There exists a constant $C>0,$ such that%
\begin{equation*}
\begin{array}{l}
\left\vert f(t,y_{1},z_{1},k_{1}(\cdot ),y_{2},z_{2},k_{2}(\cdot ))-f(t,%
\tilde{y}_{1},\tilde{z}_{1},\tilde{k}_{1}(\cdot ),\tilde{y}_{2},\tilde{z}%
_{2},\tilde{k}_{2}(\cdot ))\right\vert \\
\leq C\left( \left\vert y_{1}-\tilde{y}_{1}\right\vert +\left\vert y_{2}-%
\tilde{y}_{2}\right\vert +\left\vert z_{1}-\tilde{z}_{1}\right\vert +\int_{%
\mathbb{R}
_{0}}\left\vert k_{1}\left( \zeta \right) -\tilde{k}_{1}\left( \zeta \right)
\right\vert \nu (d\zeta )\right. \\
+\text{ }\left. \int_{
\mathbb{R}
_{0}}\left\vert k_{2}\left( \zeta \right) -\tilde{k}_{2}\left( \zeta \right)
\right\vert \nu (d\zeta )\right) \text{ } \\
\text{a.e. }t,y_{1},\tilde{y}_{1},y_{2},\tilde{y}_{2},z_{1},\tilde{z}%
_{1},k_{1}(\cdot ),\tilde{k}_{1}(\cdot ),k_{2}(\cdot ),\tilde{k}_{2}(\cdot ).%
\end{array}%
\end{equation*}
\end{itemize}

There are real numbers $\beta ,C$ for sufficiently small $\epsilon ,$ we
have that
\begin{equation*}
\beta >\frac{6C^{2}}{\epsilon }+\frac{1}{2}.
\end{equation*}

We give now, the main Theorem of this section which is the extension to the
delay case and to mean-field of Theorem $3.1$ in \cite{HOP}.

\begin{theorem}
Under the above Assumption (H3), there exists a unique solution $(Y,Z,K)$ of
a MFDBSDE $\left( \ref{1.1}\right) ,$ such that%
\begin{equation}
\begin{array}{l}
\mathbb{E}\left[ \underset{t\geq 0}{\sup }e^{\beta t}\left\vert
Y(t)\right\vert ^{2}+\int_{0}^{\infty }e^{\beta t}\left( \left\vert
Y(t)\right\vert ^{2}+\left\vert Z(t)\right\vert ^{2}+\int_{
\mathbb{R}
_{0}}\left\vert K(t,\zeta )\right\vert ^{2}\nu \left( d\zeta \right) \right)
dt\right] \\
\leq c\mathbb{E}\left[ \int_{0}^{\infty }e^{\beta t}\left\vert
f(t,0,0,0,0,0,0)\right\vert ^{2}dt\right] <\infty .%
\end{array}
\label{1.3}
\end{equation}
\end{theorem}

\begin{proof}
\textbf{1.Uniqueness}

We assume that we have two solutions, $(Y_{1},Z_{1},K_{1}),$ $%
(Y_{2},Z_{2},K_{2}),$ and we note that
\begin{equation*}
(\bar{Y},\bar{Z},\bar{K})=(Y_{1}-Y_{2},Z_{1}-Z_{2},K_{1}-K_{2}).
\end{equation*}

In what follows, we use the simplified notation:
\begin{equation*}
f(t)=f(t,Y_{t},Z_{t},K_{t}(\cdot ),\mathbb{E}[Y(t)],\mathbb{E}[Z(t)],\mathbb{
E}[K(t,\cdot )]).
\end{equation*}
Applying It\^{o}'s formula to $e^{\beta t}\left\vert \bar{Y}(t)\right\vert
^{2},$ we get for all $%
\begin{array}{c}
0<t<T\text{ (constant)}<\infty ,%
\end{array}%
$%
\begin{eqnarray}
&&%
\begin{array}{l}
e^{\beta T}\left\vert \bar{Y}(T)\right\vert ^{2}-e^{\beta t}\left\vert \bar{Y%
}(t)\right\vert ^{2}\leq \int_{t}^{T}e^{\beta s}\left\{ \beta \left\vert
\bar{Y}(s)\right\vert ^{2}+\left\vert \bar{Z}(s)\right\vert ^{2}+\int_{%
\mathbb{R}
_{0}}\left\vert \bar{K}(s,\xi )\right\vert ^{2}\nu \left( d\xi \right)
\right\} ds \\
-2\int_{t}^{T}e^{\beta s}\left\vert \bar{Y}(s)\right\vert \cdot \left\vert
\bar{f}(s)\right\vert ds+2\int_{t}^{T}e^{\beta s}\left\vert \bar{Y}%
(s)\right\vert \cdot \left\vert \bar{Z}(s)\right\vert dB(s)%
\end{array}
\notag \\
&&%
\begin{array}{c}
+2\int_{t}^{T}\int_{
\mathbb{R}
_{0}}e^{\beta s}\left\vert \bar{Y}(s)\right\vert \cdot \left\vert \bar{K}%
(s,\xi )\right\vert \tilde{N}\left( ds,d\xi \right) .%
\end{array}
\label{1.4}
\end{eqnarray}

After taking expectation, using Lipschitz condition and the fact that $%
\begin{array}{c}
2ab\leq \frac{a^{2}}{\epsilon }+\epsilon b^{2},%
\end{array}%
$ for all $a,b$ $\in
\mathbb{R}
,$ we obtain%
\begin{align}
& \mathbb{E}\left[ e^{\beta t}\lvert \bar{Y}(t)\rvert ^{2}\right] +\mathbb{E}%
\left[ \int_{t}^{T}e^{\beta s}\left\{ \beta \left\vert \bar{Y}(s)\right\vert
^{2}+\lvert \bar{Z}(s)\rvert ^{2}+\int_{%
\mathbb{R}
_{0}}\left\vert \bar{K}(s,\xi )\right\vert ^{2}\nu \left( d\xi \right)
\right\} ds\right]  \notag \\
& =\mathbb{E}\left[ e^{\beta T}\lvert \bar{Y}(T)\rvert ^{2}\right] +2\mathbb{%
E}\left[ \int_{t}^{T}e^{\beta s}\left\vert \bar{Y}(s)\right\vert \cdot
\left\vert \bar{f}(s)\right\vert ds\right]  \notag \\
& \leq \mathbb{E}\left[ e^{\beta T}\lvert \bar{Y}(T)\rvert ^{2}\right] +%
\dfrac{6C^{2}}{\epsilon }\mathbb{E}\left[ \int_{t}^{T+r}e^{\beta
s}\left\vert \bar{Y}(s)\right\vert ^{2}ds\right]  \notag \\
& +6\epsilon (1+e^{-\beta r})\mathbb{E}\left[ \int_{t}^{T+r}e^{\beta
s}\left\vert \bar{Y}(s)\right\vert ^{2}ds\right] +6\epsilon \mathbb{E}\left[
\int_{t}^{T+r}e^{\beta s}\left( \mathbb{E}\left[ \left\vert \bar{Y}%
(s)\right\vert \right] \right) ^{2}ds\right]  \notag \\
& +6\epsilon (1+e^{-\beta r})\mathbb{E}\left[ \int_{t}^{T+r}e^{\beta
s}\left\vert \bar{Z}(s)\right\vert ^{2}ds\right] +6\epsilon \mathbb{E}\left[
\int_{t}^{T+r}e^{\beta s}\left( \mathbb{E}\left[ \left\vert \bar{Z}%
(s)\right\vert \right] \right) ^{2}ds\right]  \notag \\
& +6\epsilon (1+e^{-\beta r})\mathbb{E}\left[ \int_{t}^{T+r}\int_{%
\mathbb{R}
_{0}}e^{\beta s}\left\vert \bar{K}(s,\xi )\right\vert ^{2}\nu \left( d\xi
\right) ds\right]  \notag \\
& +6\epsilon \mathbb{E}\left[ \int_{t}^{T+r}e^{\beta s}\left( \mathbb{E}%
\left[ \int_{%
\mathbb{R}
_{0}}\left\vert \bar{K}(s,\xi )\right\vert ^{2}\nu \left( d\xi \right) %
\right] \right) ^{2}ds\right] ,  \label{1.5}
\end{align}%
where we have used by changing of variables: $u:=s+r$%
\begin{eqnarray}
\int_{t}^{T}e^{\beta s}\left\vert \bar{Y}(s+r)\right\vert ^{2}ds
&=&\int_{t+r}^{T+r}e^{\beta (u-r)}\left\vert \bar{Y}(u)\right\vert ^{2}du
\label{1.6} \\
&=&e^{-\beta r}\int_{t+r}^{T+r}e^{\beta u}\left\vert \bar{Y}(u)\right\vert
^{2}du  \notag \\
&\leq &e^{-\beta r}\int_{t}^{T+r}e^{\beta s}\left\vert \bar{Y}(s)\right\vert
^{2}ds,  \notag
\end{eqnarray}%
and similarly for $Z$ and $K.$

The fact that
\begin{equation}
\begin{array}{c}
\mathbb{E}\left[ \int_{t}^{T+r}e^{\beta s}\left( \mathbb{E}\left[ \left\vert
\bar{Y}(s)\right\vert \right] \right) ^{2}ds\right] \leq \mathbb{E}\left[
\int_{t}^{T+r}e^{\beta s}\left\vert \bar{Y}(s)\right\vert ^{2}ds\right] ,%
\text{ where }\bar{Y}=\bar{Y},\bar{Z},\bar{K},%
\end{array}%
\end{equation}%
gives that%
\begin{align}
& \mathbb{E}\left[ e^{\beta t}\lvert \bar{Y}(t)\rvert ^{2}\right] +\mathbb{E}%
\left[ \int_{t}^{T}e^{\beta s}\left\{ \beta \left\vert \bar{Y}(s)\right\vert
^{2}+\lvert \bar{Z}(s)\rvert ^{2}+\int_{%
\mathbb{R}
_{0}}\left\vert \bar{K}(s,\xi )\right\vert ^{2}\nu \left( d\xi \right)
\right\} ds\right]  \notag \\
& \leq \mathbb{E}\left[ e^{\beta T}\lvert \bar{Y}(T)\rvert ^{2}\right]
+\left( \dfrac{6C^{2}}{\epsilon }+6\epsilon (2+e^{-\beta r})\right) \mathbb{E%
}\left[ \int_{t}^{T+r}e^{\beta s}\left\vert \bar{Y}(s)\right\vert ^{2}ds%
\right]  \notag \\
& +6\epsilon (2+e^{-\beta r})\mathbb{E}\left[ \int_{t}^{T+r}e^{\beta
s}\left\vert \bar{Z}(s)\right\vert ^{2}ds\right]  \notag \\
& +6\epsilon (2+e^{-\beta r})\mathbb{E}\left[ \int_{t}^{T+r}\int_{%
\mathbb{R}
_{0}}e^{\beta s}\left\vert \bar{K}(s,\xi )\right\vert ^{2}\nu \left( d\xi
\right) ds\right] .
\end{align}

Taking $\epsilon $ such $6\epsilon (2+e^{-\beta r})<\dfrac{1}{2},$ then
since $\beta >\dfrac{6C^{2}}{\epsilon }+\dfrac{1}{2},$ we have for $t<T,$%
\begin{equation}
\mathbb{E}\left[ e^{\beta t}\lvert \bar{Y}(t)\rvert ^{2}\right] \leq \mathbb{%
E}\left[ e^{\beta T}\lvert \bar{Y}(T)\rvert ^{2}\right] .
\end{equation}

Replacing $\beta $ by $\beta ^{\prime }$ with, $\beta >\beta ^{\prime }>%
\dfrac{6C^{2}}{\epsilon }+\dfrac{1}{2},$ then
\begin{equation}
\mathbb{E}\left[ e^{\beta ^{\prime }t}\lvert \bar{Y}(t)\rvert ^{2}\right]
\leq e^{(\beta ^{\prime }-\beta )T}\mathbb{E}\left[ e^{\beta T}\lvert \bar{Y}%
(T)\rvert ^{2}\right] .  \label{1.7}
\end{equation}

The second factor on the right hand side is bounded by condition $\left( \ref%
{1.3}\right) $ as $T\rightarrow \infty ,$ while the first factor tends to $%
0. $

The uniqueness is proved.

\textbf{2.Existence}

The existence proof is given into two steps, for all $n\in
\mathbb{N}
.$ Let us construct the solution $(Y^{n},Z^{n},K^{n})$ of the infinite
horizon DBSDE
\begin{equation*}
\begin{array}{c}
Y^{n}(t)=\int_{t}^{n}f(s,Y_{s}^{n},Z_{s}^{n},K_{s}^{n}(\cdot ),\mathbb{E}%
\left[ Y^{n}(s)\right] ,\mathbb{E}\left[ Z^{n}(t)\right] ,\mathbb{E}\left[
K^{n}(s,\cdot )\right] )ds \\
-\int_{t}^{n}Z^{n}(s)dB(s)-\int_{t}^{n}\int_{%
\mathbb{R}
_{0}}K^{n}(s,\xi )\tilde{N}\left( ds,d\xi \right) ;t\geq 0,%
\end{array}%
\end{equation*}%
as follows:

\begin{itemize}
\item \textbf{Step 1:} It is decomposed into two cases according to $t.$
\end{itemize}

\textbf{Case 1:} Let $0\leq t\leq n,$ then $(Y^{n},Z^{n},K^{n})$ is defined
as a solution of the DBSDE in $[0,n]$ as follows:%
\begin{equation*}
\begin{array}{c}
Y^{n}(t)=\int_{t}^{n}f(s,Y_{s}^{n},Z_{s}^{n},K_{s}^{n}(\cdot ),\mathbb{E}%
\left[ Y^{n}(s)\right] ,\mathbb{E}\left[ Z^{n}(t)\right] ,\mathbb{E}\left[
K^{n}(s,\cdot )\right] )ds \\
-\int_{t}^{n}Z^{n}(s)dB(s)-\int_{t}^{n}\int_{
\mathbb{R}
_{0}}K^{n}(s,\xi )\tilde{N}\left( ds,d\xi \right) ;0\leq t\leq n.%
\end{array}%
\end{equation*}

\textbf{Case 2:} For $t\geq n,$ $(Y^{n},Z^{n},K^{n})$ is defined by%
\begin{equation*}
\begin{array}{ll}
Y^{n}(t):=0, & t>n, \\
Z^{n}(t):=0, & t>n, \\
K^{n}(t,\cdot ):=0, & t>n.%
\end{array}%
\end{equation*}

We will first establish an a priori estimate for the sequence $%
(Y^{n},Z^{n},K^{n})$.

Adding and subtracting $\mathbb{E}\left[ \int_{t}^{\infty }e^{\beta
s}\left\vert Y(s)\right\vert \cdot \left\vert f(s,0,0,0,0,0,0)\right\vert ds%
\right] $ in $\left( \ref{1.2}\right) $ and by using $2ab\leq \frac{a^{2}}{%
\epsilon }+\epsilon b^{2},$ for all $a,b$ $\in
\mathbb{R}
,$ we get%
\begin{equation*}
\begin{array}{l}
\mathbb{E}\left[ e^{\beta t}\left\vert Y^{n}(t)\right\vert ^{2}\right] +%
\mathbb{E}\left[ \int_{t}^{\infty }e^{\beta s}\left\{ \bar{\beta}\left\vert
Y^{n}(s)\right\vert ^{2}+\bar{\alpha}\left\vert Z^{n}(s)\right\vert ^{2}+%
\bar{\gamma}\int_{
\mathbb{R}
_{0}}\left\vert K(s,\xi )\right\vert ^{2}\nu \left( d\xi \right) \right\} ds%
\right] \\
\leq \dfrac{1}{\epsilon }\mathbb{E}\left[ \int_{t}^{\infty }e^{\beta
s}\left\vert f(s,0,0,0,0,0,0)\right\vert ^{2}ds\right]%
\end{array}%
\end{equation*}%
with
\begin{equation}
\begin{array}{l}
\bar{\beta}:=\beta -\dfrac{6C^{2}}{\epsilon }-\dfrac{1}{2}>0, \\
\bar{\alpha}:=1-6\epsilon (2+e^{-\beta r})>0, \\
\bar{\gamma}:=1-6\epsilon (2+e^{-\beta r})>0,%
\end{array}
\label{*}
\end{equation}%
where we have taking that $6\epsilon (2+e^{-\beta r})<1.$

Using $\left( \ref{*}\right) $ and the martingale inequality, we have
\begin{equation*}
\begin{array}{l}
\mathbb{E}\left[ \underset{t\geq s}{\sup }e^{\beta t}\left\vert
Y^{n}(t)\right\vert ^{2}\right] +\mathbb{E}\left[ \int_{s}^{\infty
}e^{\beta r}\left\{ \left\vert Y^{n}(r)\right\vert ^{2}+\left\vert
Z^{n}(r)\right\vert ^{2}+\int_{
\mathbb{R}
_{0}}\left\vert K^{n}(r,\xi )\right\vert ^{2}\nu \left( d\xi \right)
\right\} dr\right] \\
\leq \dfrac{1}{\epsilon }\mathbb{E}\left[ \int_{s}^{\infty }e^{\beta
r}\left\vert f(r,0,0,0,0,0,0)\right\vert ^{2}dr\right] .%
\end{array}%
\end{equation*}

\begin{itemize}
\item \textbf{Step 2:} Let $m>n,$ define
\begin{equation*}
\begin{array}{l}
\triangle Y(t)=Y^{m}(t)-Y^{n}(t), \\
\triangle Z(t)=Z^{m}(t)-Z^{n}(t), \\
\triangle K(t,\cdot )=K^{m}(t,\cdot )-K^{n}(t,\cdot ).%
\end{array}%
\end{equation*}
\end{itemize}

According to $t,$ we have also two cases.

\textbf{Case 1:} For $n\leq t\leq m,$ we have%
\begin{eqnarray*}
\triangle Y(t) &=&\int_{t}^{m}f(s,Y_{s}^{m},Z_{s}^{m},K_{s}^{m}(\cdot ),%
\mathbb{E}\left[ Y^{m}(s)\right] ,\mathbb{E}\left[ Z^{m}(s)\right] ,\mathbb{E%
}\left[ K^{m}(s,\cdot )\right] )ds \\
&&-\int_{t}^{m}\triangle Z(s)dB(s)-\int_{t}^{m}\int_{%
\mathbb{R}
_{0}}\triangle K(s,\xi )\tilde{N}\left( ds,d\xi \right) .
\end{eqnarray*}

Then by It\^{o}'s formula, we get%
\begin{equation*}
\begin{array}{l}
\mathbb{E}\left[ e^{\beta t}\left\vert \triangle Y(t)\right\vert ^{2}\right]
+\mathbb{E}\left[ \int_{t}^{m}e^{\beta s}\left\{ \beta \left\vert \triangle
Y(s)\right\vert ^{2}+\left\vert \triangle Z(s)\right\vert ^{2}+\int_{%
\mathbb{R}
_{0}}\left\vert K^{n}(s,\xi )\right\vert ^{2}\nu \left( d\xi \right)
\right\} ds\right] \\
=2\mathbb{E}\left[ \int_{t}^{m}e^{\beta s}\left\vert \triangle
Y(s)\right\vert \cdot \right. \\
\left. \cdot \left\vert \triangle f(s,Y_{s}^{m},Z_{s}^{m},K_{s}^{m}(\cdot ),%
\mathbb{E}\left[ Y^{m}(s)\right] ,\mathbb{E}\left[ Z^{m}(s)\right] ,\mathbb{E%
}\left[ K^{m}(s,\cdot )\right] )\right\vert ds\right] \\
\leq 2C\mathbb{E}\left[ \int_{t}^{m}e^{\beta s}\left\vert Y(s)\right\vert
\cdot \left\vert \triangle f(s)\right\vert ds\right] \\
+2\mathbb{E}\left[ \int_{t}^{m}e^{\beta s}\left\vert \triangle
Y(s)\right\vert \cdot \left\vert \triangle f(s,0,0,0,0,0,0)\right\vert ds%
\right] .%
\end{array}%
\end{equation*}

We deduce, by the same argument used before, that%
\begin{equation*}
\begin{array}{l}
\mathbb{E}\left[ \underset{n\leq t\leq m}{\sup }e^{\beta t}\left\vert
\triangle Y(t)\right\vert ^{2}\right] +\mathbb{E}\left[ \int_{n}^{m}e^{%
\beta s}\left\{ \left\vert \triangle Y(s)\right\vert ^{2}+\left\vert
\triangle Z(s)\right\vert ^{2}+\int_{
\mathbb{R}
_{0}}\left\vert \triangle K(s,\xi )\right\vert ^{2}\nu \left( d\xi \right)
\right\} ds\right] \\
\leq \dfrac{1}{\epsilon }\mathbb{E}\left[ \int_{n}^{\infty }e^{\beta
s}\left\vert \triangle f(s,0,0,0,0,0,0)\right\vert ^{2}ds\right] .%
\end{array}%
\end{equation*}

The last inequality goes to zero as $n$ goes to infinity.

\textbf{Case 2:} For $t\leq n,$ we have%
\begin{eqnarray*}
\triangle Y(t) &=&\triangle Y(n)+\int_{t}^{n}\left\{
f(s,Y_{s}^{m},Z_{s}^{m},K_{s}^{m}(\cdot ),\mathbb{E}\left[ Y^{m}(s)\right] ,%
\mathbb{E}\left[ Z^{m}(s)\right] ,\mathbb{E}\left[ K^{m}(s,\cdot )\right]
)\right. \\
&&\left. -f(s,Y_{s}^{n},Z_{s}^{n},K_{s}^{m}(\cdot ),\mathbb{E}\left[ Y^{n}(s)%
\right] ,\mathbb{E}\left[ Z^{n}(s)\right] ,\mathbb{E}\left[ K^{m}(s,\cdot )%
\right] )\right\} ds \\
&&-\int_{t}^{n}\triangle Z(s)dB(s)-\int_{t}^{m}\int_{%
\mathbb{R}
_{0}}\triangle K(s,\xi )\tilde{N}\left( ds,d\xi \right) .
\end{eqnarray*}

We proceed as in $\left( \ref{1.5}\right) -\left( \ref{1.7}\right) ,$ in the
proof of uniqueness, we obtain%
\begin{eqnarray*}
\mathbb{E}\left[ e^{\beta t}\left\vert \triangle Y(t)\right\vert ^{2}\right]
&\leq &\mathbb{E}\left[ e^{\beta n}\left\vert \triangle Y(n)\right\vert ^{2}%
\right] \\
&\leq &\dfrac{1}{\epsilon }\mathbb{E}\left[ \int_{n}^{\infty }e^{\beta
s}\left\vert \triangle f(s,0,0,0,0,0,0)\right\vert ^{2}ds\right] .
\end{eqnarray*}

Then the sequence $(Y^{n},Z^{n},K^{n})$ forms a Cauchy sequence for the norm
of the space $\mathcal{L}$ in $\left( \ref{1.2}\right) $ and that the limit $%
(Y,Z,K)$ is a solution of a MFDBSDE $\left( \ref{1.1}\right) .$
\end{proof}

\end{document}